\documentclass{article}
\usepackage{graphicx} 
\usepackage{fullpage}
\usepackage{amssymb, amsmath}
\usepackage{amsthm} 
\usepackage{hyperref}
\usepackage{color}

\usepackage{url}

\usepackage{bm}
\usepackage{enumitem} 

\usepackage{tikz}
\usetikzlibrary{arrows.meta,positioning}

\newtheorem{theorem}{Theorem}
\newtheorem{lemma}[theorem]{Lemma}

\newtheorem{cor}[theorem]{Corollary}
\newtheorem{obs}[theorem]{Observation}
\newtheorem{prob}[theorem]{Problem}

\theoremstyle{definition}

\newtheorem*{theorem*}{Theorem}
\newtheorem*{lemma*}{Lemma}

\newcommand{\rank}{\operatorname{rank}}
\newcommand{\trace}{\operatorname{trace}}

\title{$1$-cross intersecting set pair systems are small}
\author{Dylan King\thanks{Department of Mathematics, California Institute of Technology, Pasadena, CA. E-mail: \texttt{dking@caltech.edu}.} \qquad Van Magnan\thanks{Department of Mathematics and Statistics, University of Vermont, Burlington, VT. E-mail: \texttt{van.magnan@uvm.edu}.} \qquad Cory Palmer\thanks{Department of Mathematical Sciences, University of Montana, Missoula, MT. E-mail: \texttt{cory.palmer@umontana.edu}. Supported by NSF grant DMS-2503179 and a grant from the Simons Foundation [SFI-MPS-TSM-00013277, CP].}}

\begin{document}

\maketitle

\begin{abstract}
    A set pair system $\{(A_i,B_i)\}_{i=1}^m$ is \emph{$1$-cross intersecting} if $|A_i \cap B_i|=0$ for all $i$ and $|A_i \cap B_j| = 1$ whenever $i \neq j$. Let $m(a,b,1)$ denote the maximum size $m$ of a $1$-cross intersecting set pair system $\{(A_i,B_i)\}_{i=1}^m$ where $|A_i| \leq a$ and $|B_i| \leq b$ for all $i$.
    We prove a conjecture of F\"uredi, Gy\'arf\'as, and Kir\'aly [\emph{Combin. Probab. Comput.} \textbf{32} (2023)] that asserts $m(n,n,1)/\binom{2n}{n} \to 0$ as $n \to \infty$.
\end{abstract}

\section{Introduction}

Let $\{(A_i,B_i)\}_{i=1}^m$ be a family of $m\geq 2$ pairs of finite sets, called a \emph{set pair system}. Such a system is \emph{cross intersecting} if $|A_i \cap B_i|=0$ for all~$i$ and $|A_i \cap B_j| \geq 1$ whenever $i\neq j$. A classical result of Bollob\'as~\cite{Bollobas1965} says that if $|A_i| \leq a$ and $|B_i| \leq b$ for all $i$, then $m \leq \binom{a+b}{a}$. Taking~$A_i$ and~$B_i$ to be complementary pairs of sets of size~$a$ and~$b$ inside~$\{1,2,\dots, a+b\}$ shows that this is best possible. 
This result has seen many applications, extensions, and alternative proofs, including the more general YBLM inequality.
See Tuza's two-part survey~\cite{Tuza1994, Tuza1996} for more on the rich history of cross intersecting set pair systems. 

The following variant of cross intersection is due to F\"uredi, Gy\'arf\'as, and Kir\'aly~\cite{FGK2023}. A set pair system $\{(A_i,B_i)\}_{i=1}^m$ is \emph{$1$-cross intersecting} if  $|A_i \cap B_i|=0$ for all~$i$ and $|A_i \cap B_j| = 1$ whenever $i \neq j$.
Let $m(a,b,1)$ denote the maximum size $m$ of a $1$-cross intersecting set pair system $\{(A_i,B_i)\}_{i=1}^m$ where $|A_i| \leq a$ and $|B_i| \leq b$ for all $i$. Focusing for a moment on the diagonal case~$a=b=n$, Bollob\'as's theorem immediately shows that~$m(n,n,1) \leq \binom{2n}{n}$ since every~$1$-cross intersecting family is also cross intersecting. F\"uredi, Gy\'arf\'as, and Kir\'aly~\cite{FGK2023} asked if there is a constant $\varepsilon >0$ such that $m(n,n,1) \leq (1-\varepsilon)\binom{2n}{n}$ for all $n\geq 2$. They also made the stronger conjecture that $m(n,n,1)/\binom{2n}{n} \to 0$ as $n \to \infty$.

The first question was answered by Holzman~\cite{Holzman2021} who proved that, for each~$n \geq 2$,
\[
m(n,n,1) \leq \frac{29}{30} \binom{2n}{n}.
\]
The coefficient in the upper bound was later improved to $\frac{5}{6}$ by Kostochka, McCourt, and Nahvi~\cite{KMN2021}.
The best lower bounds on~$m(n,n,1)$ are due to F\"uredi, Gy\'arf\'as, and Kir\'aly~\cite{FGK2023}:
\[
m(n,n,1) \geq 
\begin{cases}
5^{n/2} & \text{if $n$ even,}  \\
2 \cdot 5^{(n-1)/2} & \text{if $n$ odd}.
\end{cases}
\]
This bound is known to be optimal for $n=2,3$ (the case $n=3$ was verified by computer by Spiro).

We confirm the second conjecture of F\"uredi, Gy\'arf\'as, and Kir\'aly by showing that~$m(n,n,1)/\binom{2n}{n}$ decays exponentially.

\begin{theorem}\label{main-thm}
   Let $\{(A_i,B_i)\}_{i=1}^m$ be a $1$-cross intersecting set pair system with $|A_i| \leq a$ and $|B_i| \leq b$ for every $1 \leq i \leq m$. If $m \geq 2$, then 
   \[
   m \leq e^{\sqrt{ab}}.
   \]
\end{theorem}

\begin{cor}

\[
\lim_{n \to \infty} \frac{m(n,n,1)}{\binom{2n}{n} } = 0.
\]
    
\end{cor}

\medskip

\noindent {\bf AI Use.}
Several key ideas of this paper were developed through a sequence of prompts to Anthropic's Claude (Opus 4.8, Opus 5, Sonnet 5) and OpenAI's ChatGPT (5.2 Thinking, 5.6 Sol). Initial attempts to further improve on the relative bound of~$\frac{5}{6}$, for large~$n$ (obtained in conjunction with ChatGPT 5.6 Sol), were simplified into the existing linear algebraic argument.
The presentation in this article is that of the authors.

\section{Proof}

The proof of Theorem~$\ref{main-thm}$ is a short linear algebraic argument, requiring only two standard inequalities, which we recall now.

For non-increasing real sequences of length~$n$, $x=(x_1,\dots, x_n)$ and $y=(y_1,\dots, y_n)$, we say that $x$ \emph{majorizes} $y$ if the following hold:
\begin{enumerate}[label=(\roman*)]
    \item\label{cond-i} both sequences have the same sum, i.e.,  $\sum_{i=1}^{n} x_i = \sum_{i=1}^{n} y_i$, and
    
    \item\label{cond-ii} the partial sums satisfy $\sum_{i=1}^{k} x_i \geq \sum_{i=1}^{k} y_i$ for all $1 \leq k \leq n-1$.
\end{enumerate}

The following lemma states that concave functions respect majorization. 

\begin{lemma}[Karamata's inequality~\cite{Karamata1932}]\label{karamata}
    Suppose $x=(x_1,\dots, x_n)$ and $y=(y_1,\dots, y_n)$ are non-increasing real sequences such that $x$ majorizes $y$.
    If $f$ is a real-valued concave function, then
    \[
    \sum_{i=1}^n f(x_i) \leq \sum_{i=1}^n f(y_i).
    \]
\end{lemma}

The second tool we require is the so-called ``Rank Lemma'' which bounds the trace of a matrix in terms of its rank and Frobenius norm. In particular, for an $n \times n$ matrix $C = [c_{ij}]$, 
\[
(\trace C)^2 \leq \rank C \cdot \left(\sum_{i=1}^n \sum_{j=1}^n c_{ij}^2 \right).      
\]
See, for example, Lemma 41 of~\cite{KS} for a statement and references to combinatorial and geometric applications. We require a corollary specialized to zero-one matrices (matrices where all entries are either $0$ or $1$).
For a zero-one matrix $C$, let the \textit{weight} of $C$, denoted $w(C)$, be the number of $1$-entries in $C$. In such a matrix, every term $c_{ij}^2$ is zero or one and so we have:

\begin{lemma}[Rank Lemma]\label{rank-lemma} For a zero-one square matrix $C$,
    \[
    (\trace C)^2 \leq \rank C \cdot w(C).
    \]    
\end{lemma}

With these two tools in hand, we are now ready to prove the main result.

\begin{proof}[Proof of Theorem~\ref{main-thm}]
Fix $m\geq 2$ and let $\{(A_i,B_i)\}_{i=1}^m$ be a $1$-cross intersecting set pair system with $|A_i| \leq a$ and $|B_i| \leq b$ for every $1 \leq i \leq m$. Let 
\[
\bigcup_{i=1}^m (A_i \cup B_i)
\]
be the \emph{ground set} of the system, which, by relabeling, we identify with~$[N] = \{1,2,\dots, N\}$. For each $\ell \in [N]$, define incidence column vectors $\bm{a}_\ell,\bm{b}_\ell \in \{0,1\}^{m}$ by
\[
\bm{a}_\ell = (\mathbf{1}_{\ell \in A_i})_{i=1}^m \qquad \text{and} \qquad \bm{b}_\ell = (\mathbf{1}_{\ell \in B_i})_{i=1}^m.
\]

We will denote by $d_A(\ell)$ and $d_B(\ell)$ the number of 1-entries in $\bm{a}_\ell$ and $\bm{b}_\ell$, respectively; these correspond to the \emph{degree} of~$\ell$ in the individual set systems $\{A_i\}_{i=1}^m$ and~$\{B_i\}_{i=1}^m$. Let $M_\ell$ be the $m \times m$ matrix resulting from the outer product $\bm{a}_\ell \bm{b}_\ell^\top$. Note that $M_\ell$ is a zero-one matrix, with weight $w(M_\ell)= d_A(\ell)d_B(\ell)$ which we abbreviate as $w_\ell$. Observe that $w_\ell$ counts the number of pairs $(A_i,B_j)$ such that $A_i \cap B_j = \{\ell\}$.

We will bound~$\sum_{\ell \in [N]}\sqrt{w_\ell}$ from above and below in order to conclude an upper bound on~$m$.

The upper bound follows by Cauchy-Schwarz and the hypothesis that $|A_i|\leq a$ and $|B_i|\leq b$ for all $i \in [m]$;
\begin{align}\label{UB}
\sum_{\ell\in [N]}\sqrt{w_\ell}=\sum_{\ell\in [N]}\sqrt{d_A(\ell)d_B(\ell)}\leq \sqrt{\left(\sum_{\ell\in [N]}d_A(\ell)\right)\left(\sum_{\ell\in [N]}d_B(\ell)\right)}\leq \sqrt{(ma)(mb)}=m\sqrt{ab}.
\end{align}

To obtain a lower bound, we will need to consider the way the intersections~$A_i \cap B_j$ can be spread across the ground set~$[N]$.
Let $J$ and $I$ be the $m\times m$ all-ones and identity matrices. The key step, and only place in the proof where we use the $1$-cross intersecting property, is the following lemma.

\begin{lemma}\label{matrix-decomp}
    The matrices $M_\ell$ form an entrywise partition of the matrix $J-I$. In particular,
    \begin{equation}\label{matrix-sum}
    \sum_{\ell\in [N]} M_\ell = J-I.   
    \end{equation}
    Consequently,
    \begin{equation}\label{total-ones}
    \sum_{\ell \in [N]} w_\ell = m(m-1).    
    \end{equation}

\end{lemma}

\begin{proof}
    By the definition of the outer product  $\bm{a}_\ell \bm{b}_\ell^{\top}$, 
    the $(i,j)$ entry of $M_\ell$ is $1$ if $\ell \in A_i \cap B_j$ and $0$ otherwise. Hence the $(i,j)$ entry of $\sum_{\ell\in [N]}M_\ell$ counts the number of elements in $A_i\cap B_j$. By the $1$-cross intersecting property, we know that $|A_i \cap B_j|=1$ when $i\neq j$ and $|A_i \cap B_j|=0$ when $i=j$.
\end{proof}

Hence the total weight,~$m(m-1)$, is distributed across the~$N$ matrices~$M_\ell$.
By relabeling elements if necessary, we may assume that the sequence of weights~$w_1 \geq \dots \geq w_N$ is non-increasing.
We are interested in finding a majorizing sequence for the~$w_\ell$, and to do so we first examine partial sums of the matrices~$M_\ell$. For $1\leq k \leq N$, define
\[
C_k = J-\sum_{\ell=1}^kM_\ell.
\]
The diagonal entries of each~$M_\ell$ are~$0$, so $\trace C_k = m$.
By substitution into \eqref{matrix-sum}, we see that $w(C_k)=m+\sum _{\ell=k+1}^Nw_\ell$.
The matrices $J$ and $M_\ell=\bm{a}_\ell\bm{b}_\ell^\top$ all have rank at most~$1$, so by subadditivity of rank
\[
\rank C_k
\leq \rank J + \sum_{\ell=1}^k\rank M_\ell \leq k+1.
\]
By Lemma~\ref{matrix-decomp},~$C_k$ is a zero-one matrix and so we may apply the Rank Lemma (Lemma~\ref{rank-lemma}) to obtain
    \[
    m^2 \leq (k+1)\left(m+\sum_{\ell=k+1}^N w_\ell\right).
    \]
       Rearranging terms and substituting \eqref{total-ones} yields
    \begin{equation}\label{partial-sum-q}
\sum_{\ell=1}^k w_\ell\leq \frac{km^2}{k+1}.
\end{equation}

With this bound in hand, we can readily find a sequence $w^*=(w_1^*, \dots, w_N^*)$ that majorizes $w=(w_1, \dots, w_N)$. 
Taking $k=N$ in \eqref{partial-sum-q} and using \eqref{total-ones}, we obtain
\[
m(m-1) \leq \frac{Nm^2}{N+1}
\]
and hence $N \geq m-1$. 
We may now define
\[
w_\ell^* = \frac{m^2}{\ell(\ell+1)} \quad(1\le \ell\le m-1),
  \qquad
  w_\ell^*=0\quad(m\le \ell\le N).
\]
One can check that the~$w_\ell^*$ are non-increasing and that
\[
\sum_{\ell=1}^k w_\ell^* =  \frac{km^2}{k+1} \quad(1\leq k\leq m-1), \qquad \sum_{\ell=1}^kw_\ell^*=\sum_{\ell=1}^{m-1}w_\ell^*=m(m-1) \quad(m \leq k \leq N), 
\]
and, in particular, $\sum_{\ell=1}^Nw_\ell^*=m(m-1)$.  Therefore, comparing with \eqref{partial-sum-q} and \eqref{total-ones}, we see that $w^*$ majorizes $w$.
Since the square-root function is concave, Karamata's inequality (Lemma~\ref{karamata}) implies
 \begin{equation}\label{q-lower-bound}
   \sum_{\ell=1}^N\sqrt{w_\ell} \geq \sum_{\ell=1}^N \sqrt{w_\ell^*} = m\sum_{\ell=1}^{m-1}\frac{1}{\sqrt{\ell(\ell+1)}}. 
 \end{equation}

Combining upper bound \eqref{UB} and lower bound \eqref{q-lower-bound}  gives
\begin{equation}\label{sum-ineq}
  \sum_{\ell=1}^{m-1}\frac{1}{\sqrt{\ell(\ell+1)}}  \leq \sqrt{ab}.  
\end{equation} 

From here it only remains to extract an upper bound for $m$, but we must be slightly delicate with calculus.
Recall that
\[
\ln \left(\frac{\ell+1}{\ell}\right) = \int_{\ell}^{\ell+1} \frac{1}{z}\, \mathrm{d}z
\]
by definition. By Cauchy-Schwarz 
\[
\left(\int_{\ell}^{\ell+1} \frac{1}{z}\, \mathrm{d}z\right)^2 \leq \left(\int_{\ell}^{\ell+1} 1 \, \mathrm{d}z\right)\left(\int_{\ell}^{\ell+1} \frac{1}{z^2}\, \mathrm{d}z\right) = 1 \left(\frac{1}{\ell} - \frac{1}{\ell+1}\right) = \frac{1}{\ell(\ell+1)}.
\]
Therefore,
\[
 \ln \left(\frac{\ell+1}{\ell}\right) \leq \frac{1}{\sqrt{\ell(\ell+1)}},
\]
and summing the above for $1 \leq \ell \leq m-1$ telescopes into
\[
\ln m \leq \sum_{\ell=1}^{m-1} \frac{1}{\sqrt{\ell(\ell+1)}}.
\]
Together with \eqref{sum-ineq}, this gives $\ln m \leq \sqrt{ab}$, i.e., $m \leq e^{\sqrt{ab}}$.
\end{proof}

\section{Concluding remarks}

It remains to further describe the asymptotic behavior of $m(n,n,1)$. In \cite{FGK2023}, the authors show via a product construction that $m(a_1+a_2,b_1+b_2,1) \geq m(a_1,b_1,1)  m(a_2,b_2,1)$. In particular, $\ln m(n,n,1)$ is superadditive and so Fekete's lemma implies that $\lim_{n \to \infty} m(n,n,1)^{1/n}$ exists.
In light of the lower bound from \cite{FGK2023} and Theorem~\ref{main-thm}, the situation is thus:
\[
\sqrt{5} \leq \lim_{n \to \infty} m(n,n,1)^{1/n} \leq e.
\]
There appears to be some slack in the proof of Theorem~\ref{main-thm}, particularly in the construction of the $w^*_\ell$ values and their use in Karamata's inequality. Thus it may be interesting to find a further exponential improvement: 

\begin{prob}
Is there an $\varepsilon >0$ such that $\lim_{n \to \infty} m(n,n,1)^{1/n} \leq e-\varepsilon$? 
\end{prob}

One route to such an improvement is to understand how the cross-intersections are distributed over the ground set.
For each ground-set element~$\ell$, the weight $w_\ell$ records exactly the number of pairs $(i,j)$ for which $A_i \cap B_j = \{\ell\}$. The proof of Theorem~\ref{main-thm} implies that extremal $1$-cross intersecting set pair systems must have highly skewed weights $w_\ell$. Such asymmetry appears in the product constructions given in \cite{FGK2023} of size $5^{n/2}$. Indeed, there are five ground set elements which are each contained in approximately $16\%$ of all the $(A_i,B_j)$ pairs, while approximately $80\%$ of the ground elements are each contained in only four pairs.

 We specialize to the case where $a=b=n$ for the following observations, but they hold for all $a$ and $b$.

\begin{obs}
 Let $m\geq 2$ and $\{(A_i,B_i)\}_{i=1}^m$ be a $1$-cross intersecting family such that $|A_i|,|B_i| \leq n$. There  exists a ground-set element $\ell$ satisfying $\displaystyle w_\ell \geq \frac{(m-1)^2}{n^2}$.   
\end{obs}

\begin{proof}
    Let $[N]$ be the ground set of $\{(A_i,B_i)\}_{i=1}^m$.
    Using \eqref{total-ones} and \eqref{UB}, 
    \[
    m(m-1)=\sum_{\ell\in [N]}w_\ell\leq \left(\max_{\ell\in [N]}\sqrt{w_\ell}\right) \sum_{\ell\in [N]}\sqrt{w_\ell}\leq \left(\max_{\ell\in [N]}\sqrt{w_\ell}\right)mn.
    \]
    Rearranging yields $\max_{\ell\in [N]}\sqrt{w_\ell}\geq \frac{m-1}{n}$. 
\end{proof}

In contrast, Proposition 1.3 of~\cite{FGK2023} shows that $m\leq N$, so the average of $w_\ell$ is at most $m-1$. 

We can also show the existence of a small set of elements which witness nearly all the intersections~$A_i \cap B_j$.

\begin{obs} Let $m\geq 2$ and $\{(A_i,B_i)\}_{i=1}^m$ be a $1$-cross intersecting family such that $|A_i|,|B_i|\leq n$. For $\varepsilon >0$, there is a set of ground elements $T$ with $|T|\leq {2n^2}/{\varepsilon}$ such that
\[
\sum_{\ell \in T} w_\ell \geq (1-\varepsilon)m(m-1).
\]
\end{obs}

\begin{proof}
Let $[N]$ be the ground set of $\{(A_i,B_i)\}_{i=1}^m$.
    Set
    \[
    \lambda=\frac{\varepsilon m(m-1)}{\sum_{\ell\in [N]}\sqrt{w_\ell}} \qquad \text{and} \qquad T=\{\ell\in [N]\,:\, \sqrt{w_\ell} \geq \lambda\}.
    \]
    If $\ell\in T$, then $\sqrt{w_\ell} \geq \lambda$, so 
    \[
    |T|\lambda \leq \sum_{\ell\in T}\sqrt{w_\ell}\leq \sum_{\ell\in [N]}\sqrt{w_\ell}.
    \]
    Therefore, using \eqref{UB},
    \[
    |T|\leq \frac{\sum_{\ell\in [N]}\sqrt{w_\ell}}{\lambda}=\frac{\left(\sum_{\ell\in [N]}\sqrt{w_\ell}\right)^2}{\varepsilon m (m-1)}\leq \frac{m^2 n^2}{\varepsilon m (m-1)} = \frac{m}{m-1} \frac{n^2}{\varepsilon}\leq \frac{2n^2}{\varepsilon},
    \]
    and thus $|T|$ has the appropriate size.
    
    Now, if $\ell\not\in T$, then $\sqrt{w_\ell} \leq \lambda$, so $w_\ell\leq \lambda \sqrt{w_\ell}$. Summing, we see
    \[
    \sum_{\ell\not\in T}w_\ell\leq \lambda \sum_{\ell\not\in T}\sqrt{w_\ell}\leq \lambda\sum_{\ell\in [N]}\sqrt{w_\ell}=\varepsilon m (m-1).
    \]
    Recalling \eqref{total-ones}, the elements in $T$ must satisfy 
    \[
    \sum_{\ell\in T}w_\ell\geq (1-\varepsilon)m(m-1). \qedhere
    \]
\end{proof}

Theorem~\ref{main-thm} also has implications for the related \emph{biclique cover problem.}
A \emph{biclique cover} of a graph $G$ is a collection of complete bipartite subgraphs (bicliques) of $G$ whose edge sets cover $E(G)$, while a \emph{biclique partition} requires that these edge sets partition $E(G)$. A biclique cover (partition) is \emph{$r$-local} if every vertex is in at most $r$ of its bicliques. The minimum $r$ for which there is an $r$-local biclique cover (partition) of a graph $G$ is denoted $\mathrm{lbc}(G)$ and $\mathrm{lbp}(G)$, respectively.
Since every biclique partition is a biclique cover, $\mathrm{lbc}(G) \leq \mathrm{lbp}(G)$. Pinto~\cite{Pinto2014} investigated how large a gap between these parameters is possible, and  
showed that for every $t \geq 2$, there exists a graph $G$ that admits a $2$-local cover by at most $t$ bicliques such that
\begin{equation}\label{pinto-bound}
    \mathrm{lbp}(G) \geq \frac{1}{2} \log_2 \left(\frac{t-1}{3}\right).
\end{equation}
 Pinto's argument relies on estimating the $r$-local biclique partition number of the 
\emph{crown graph} $H_m$: the complete bipartite graph $K_{m,m}$ with a perfect matching removed. 
In \cite{FGK2023}, the authors show the following equivalence
\[
\textrm{lbp}(H_m) \leq n \qquad \iff \qquad m \leq m(n,n,1).
\]
Therefore, Theorem~\ref{main-thm} implies $\ln m \leq \mathrm{lbp}(H_m)$. Substituting this bound into Pinto's argument improves the leading coefficient in \eqref{pinto-bound} from $1/2$ to $\ln 2$.

\end{document}